\documentclass[11pts, a4]{article}
\usepackage{amsmath,amssymb,amsthm}
\usepackage{amsfonts}
\usepackage{amssymb}
\usepackage{amsthm}
\usepackage{newlfont}
\usepackage[all,cmtip]{xy}
\newtheorem{Theorem}{Theorem}[section]
\newtheorem{definition}[Theorem]{Definition}
\newtheorem{proposition}[Theorem]{Proposition}
\newtheorem{corollary}[Theorem]{Corollary}
\newtheorem{lemma}[Theorem]{Lemma}
\newtheorem{example}[Theorem]{Example}
\newtheorem{remark}[Theorem]{Remark}

\newcommand{\K}{\Bbb K}

\newcommand{\Z}{\Bbb Z}

\def\NL{\hfill\break}

\newcommand{\set}[1]{\left\{#1\right\}}
\newcommand{\parth}[1]{\left(#1\right)}

\title{Skew polynomial rings over idempotent reflexive and abelian rings}
\author{Mohamed Louzari}

\begin{document}
\date{}
\maketitle
\centerline{Department of mathematics, Faculty of sciences}

\centerline{Abdelmalek Essaadi University}

\centerline{B.P. 2121 Tetouan, Morocco}

\centerline{mlouzari@yahoo.com}

\begin{abstract}
Let $R$ be a ring and $\sigma$ an endomorphism of $R$. In this note, we show that skew polynomial rings and skew power series rings over idempotent reflexive rings are idempotent reflexive when $\sigma(Re)\subseteq Re$ for all $e^2=e\in R$. In addition, we have the equivalence for $\sigma$-compatible rings. It is also proved that, if $R$ satisfies the condition $(\mathcal{C_{\sigma}})$ then skew polynomial rings and skew power series rings over abelian rings are abelian.

\end{abstract}

\footnote[0]{\NL 2010 Mathematics Subject Classification. 16U80, 16S36
\NL Key words and phrases. Idempotent reflexive rings, Abelian rings, Skew polynomial rings, Skew power series rings, Ore extensions, $\sigma$-idempotent reflexive rings, $\sigma$-abelian rings}

\section{Introduction}

Throughout this paper, $R$ denotes an associative ring with unity. Let $Id(R)$ be the set of all idempotent elements of $R$. A ring $R$ is called abelian if all idempotent of it are central (i.e., $re=er$ for all $r\in R$ and $e\in Id(R)$).  Following Hashemi and Moussavi \cite{hashemi/quasi}, a ring $R$ is $\sigma$-{\it compatible} if for each $a,b\in R$, $ab=0$ if and only if $a\sigma(b)=0$. Agayev et al. \cite{agayev/2009}, introduced $\sigma$-abelian rings as a generalization of abelian rings. A ring $R$ is called $\sigma$-abelian if it is abelian and $\sigma$-compatible. A left ideal $I$ is said to be {\it reflexive} \cite{mason}, if $aRb\subseteq I$ implies $bRa\subseteq I$ for $a,b\in R$. A ring $R$ is called {\it reflexive} if $0$ is a reflexive ideal. The reflexive property for ideals was first studied by Mason \cite{mason}, this concept was generalized by Kim and Baik \cite{kim/2005,kim/2006}, and they introduced the concept of idempotent reflexive right ideals and rings. A left ideal $I$ is called {\it idempotent reflexive} \cite{kim/2005}, if $aRe\subseteq I$ implies $eRa\subseteq I$ for $a,e=e^2\in R$. A ring $R$ is called {\it idempotent reflexive} if $0$ is an idempotent reflexive ideal. Also, there is an example of an idempotent reflexive ring which is not reflexive \cite[Example 5]{kim/2005}. Kwak and Lee \cite{kwak/2012}, introduced the concept of left and right idempotent reflexive rings. A ring $R$ is called {\it left idempotent reflexive} if $eRa=0$ implies $aRe=0$ for $a,e=e^2\in R$. Right idempotent reflexive rings are defined similarly. If a ring $R$ is left and right idempotent reflexive then it is called an idempotent reflexive ring. It can be easily checked that every reflexive ring is an idempotent reflexive ring. Abelian rings (i.e., its idempotents are central) are idempotent reflexive and hence semicommutative rings are idempotent reflexive since semicommutative rings are abelian. For a subset $X$ of $R$, $r_R(X)=\{a\in R|Xa=0\}$ and $\ell_R(X)=\{a\in R|aX=0\}$ will stand for the right and the left annihilator of $X$ in $R$ respectively. Recall that a ring $R$ is satisfying the condition $(\mathcal{C_{\sigma}})$ if whenever $a\sigma(b)=0$ with $a,b\in R$, then $ab=0$ \cite{louzari}. Clearly, $\sigma$-compatible rings satisfy the condition $(\mathcal{C_{\sigma}})$.
\par For an endomorphism $\sigma$ of a ring $R$, a {\it skew polynomial} ring (also called an {\it Ore extension of endomorphism type}) $R[x;\sigma]$ of $R$ is the ring obtained by giving the polynomial ring over $R$ with the new multiplication $xr=\sigma(r)x$ for all $r\in R$. Also, a {\it skew power series ring} $R[[x;\sigma]]$ is the ring consisting of all power series of the form $\sum_{i=0}^{\infty}a_ix^i\;(a_i\in R)$, which are multiplied using the distributive law and the Ore commutation rule $xa=\sigma(a)x$, for all $a\in R$.
\par In this paper, we study skew polynomial rings and skew power series rings over idempotent reflexive rings and abelian rings. Also, we used the idea of Agayev et al. \cite{agayev/2009} to introduce the concept of right (resp., left) $\sigma$-idempotent reflexive rings which is a generalization of right (resp., left) idempotent reflexive rings and $\sigma$-abelian rings. Agayev et al. \cite{agayev/2009}, studied abelian and $\sigma$-abelian properties on skew polynomial rings. Also, Kwak and Lee \cite{kwak/2012}, investigated reflexive and idempotent reflexive properties on ordinary polynomial extensions. In this work, we continue studding abelian and idempotent reflexive properties on skew polynomial and skew power series extensions.

\section{Idempotent reflexive property}              

In \cite[Theorem 3.1]{kwak/2013}, a ring $R$ is idempotent reflexive if and only if $R[x]$ is idempotent reflexive if and only if $R[[x]]$ is idempotent reflexive. One might expect if idempotent reflexive property can be preserved under skew polynomial and skew power series extensions.

\begin{proposition}\label{prop1}Let $R$ be a ring and $\sigma$ an endomorphism of $R$ such that $\sigma(Re)\subseteq Re$ for all $e\in Id(R)$.
\par$(1)$ If $R$ is left idempotent reflexive then $R[x,\sigma]$ and $R[[x,\sigma]]$ are left idempotent reflexive.
\par$(2)$ If $R[x,\sigma]$ or $R[[x,\sigma]]$ is right idempotent reflexive then $R$ is right idempotent reflexive.
\end{proposition}

\begin{proof}$(1)$ Let $e(x)=\sum_{j=0}^me_jx^j\in Id(R[x,\sigma])$ and $f(x)=\sum_{i=0}^nf_ix^i\in R[x,\sigma]$. Assume that $e(x)R[x,\sigma]f(x)=0$, then $e(x)Rf(x)=0$. Thus for any $b\in R$, we have the following system of equations.
$$e_0bf_0=0\eqno(0)$$
$$e_0bf_1+e_1\sigma(bf_0)=0\qquad\qquad\;\;\;\eqno(1)$$
$$e_0bf_2+e_1\sigma(bf_1)+e_2\sigma^2(bf_0)=0\qquad\qquad\qquad\qquad\qquad\eqno(2)$$
$$\cdots$$
From equation $(0)$, $f_0\in r_R(e_0R)$. Let $s\in R$ and take $b=e_0s$ in equation $(1)$. Then $e_0sf_1+e_1\sigma(e_0sf_0)=0$, but $f_0\in r_R(e_0R)$.  So $e_0sf_1=0$, hence  $f_1\in r_R(e_0R)$. Now, in equation $(2)$, take $b=e_0s$ for each $s\in R$. Then $e_0sf_2+e_1\sigma(e_0sf_1)+e_2\sigma^2(e_0sf_0)=0$, since $f_0,f_1\in r_R(e_0R)$ so $e_0sf_2=0$. Hence $f_2\in r_R(e_0R)$. Continuing this procedure yields $f_i\in r_R(e_0R)$ for all $i=0,1,2,\cdots,n$. Hence $e_0Rf_i=0$ for all $i=0,1,2,\cdots,n$. Thus $f_iRe_0=0$ for all $i=0,1,\cdots,n$ because $R$ is left idempotent reflexive. On the other hand, $$f(x)Re_0=f_0Re_0+f_1xRe_0+\cdots+f_nx^nRe_0$$
$$\qquad\qquad\qquad\quad\;=f_0Re_0+f_1\sigma(Re_0)x+\cdots+f_n\sigma^n(Re_0)x^n$$
Since $\sigma(Re_0)\subseteq Re_0$, it follows that $f(x)Re_0\subseteq \sum_{i=0}^nf_iRe_0=0$, then $f(x)Re_0=0$. With the same method as above, we obtain $$f(x)Rxe_0=f(x)Rx^2e_0=\cdots=f(x)Rx^ne_0=0.$$ So $f(x)R[x,\sigma]e_0=0$. However, since $e(x)^2=e(x)$, we have $e_0^2=e_0,\; e_0e_1+e_1\sigma(e_0)=e_1,\: \cdots$. Using the hypothesis $``\sigma(Re_0)\subseteq Re_0"$, we have $e_j\in Re_0R$ for any $1\leq j\leq m$. Thus from the fact $f(x)R[x;\sigma]e_0=0$, we have $f(x)R[x;\sigma]e_j=0$, and therefore $f(x)R[x;\sigma]e(x)=0$. We use similar argument for the case of skew power series rings.
\par$(2)$ Suppose that $R[x;\sigma]$ (resp., $R[[x;\sigma]]$) is right idempotent reflexive. If $aRe=0$ for some $e^2=e,a\in R$ then $aR[x;\sigma]e=0$, by the hypothesis ``$\sigma(Re)\subseteq Re$". Since $R[x;\sigma]$ (resp., $R[[x;\sigma]]$) is right idempotent reflexive, we get $eR[x;\sigma]a=0$, then $eRa=0$.  Therefore $R$ is right idempotent reflexive.
\end{proof}

We can easily see that, if a ring $R$ is $\sigma$-compatible then $\sigma(e)=e$ for all $e\in Id(R)$, and so the condition ``$\sigma(Re)\subseteq Re$ for all $e\in Id(R)$", is satisfied.
\begin{proposition}\label{prop2}Let $R$ be a ring and $\sigma$ an endomorphism of $R$ such that $R$ is $\sigma$-compatible.
\par$(1)$ If $R$ is right idempotent reflexive then $R[x,\sigma]$ and $R[[x,\sigma]]$ are right idempotent reflexive.
\par$(2)$ If $R[x,\sigma]$ or $R[[x,\sigma]]$ is left idempotent reflexive then $R$ is left idempotent reflexive.
\end{proposition}

\begin{proof}$(1)$ Assume that $R$ is right idempotent reflexive and let $e(x)=\sum_{i=0}^ne_ix^i\in Id(R[x,\sigma])$ and $f(x)=\sum_{j=0}^mf_jx^j\in R[x,\sigma]$ such that $f(x)R[x,\sigma]e(x)=0$. Thus we have the following system of equations where $b$ is an arbitrary element of $R$.
$$f_0be_0=0\eqno(0)$$
$$f_0be_1+f_1\sigma(be_0)=0\qquad\qquad\;\;\;\eqno(1)$$
$$f_0be_2+f_1\sigma(be_1)+f_2\sigma^2(be_0)=0\qquad\qquad\qquad\qquad\qquad\eqno(2)$$
$$\cdots$$
Equation $(0)$ yields $f_0\in \ell_R(Re_0)$. In equation $(1)$ substitute $se_0$ for $b$ to obtain $f_1\sigma(se_0)=0$, but $R$ is $\sigma$-compatible then $f_1se_0=0$ and so $f_1\in \ell_R(Re_0)$. Continuing this procedure yields $f_i\in \ell_R(Re_0)$ for all $i=0,1,\cdots,n$.
Since $R$ is right idempotent reflexive, it follows that $e_0Rf_i=0$ for all $i=0,1,\cdots,n$. We claim that $e_0R[x,\sigma]f(x)=0$. Observe that $e_0Rf(x)=\sum_{j=0}^me_0Rf_jx^j=0$ and $e_0Rxf(x)=\sum_{j=0}^me_0R\sigma(f_j)x^{j+1}=0$
because $R$ is $\sigma$-compatible. Also, we have $e_0Rx^kf(x)=0$ for all nonnegative integers $k$. Consequently $e_0\varphi(x)f(x)=0$ for all $\varphi(x)\in R[x,\sigma]$. Thus $e_0R[x,\sigma]f(x)=0$. As, in the proof of Proposition \ref{prop1}, we have $e_i\in Re_0R$ for any $1\leq i\leq n$. It follows that $e(x)R[x,\sigma]f(x)=0$. Therefore, $R[x,\sigma]$ is right idempotent reflexive.
\par$(2)$ As in the proof of Proposition \ref{prop1}(2), it suffices to replace the hypothesis ``$\sigma(Re)\subseteq Re$" by ``$\sigma$-compatible".
\end{proof}

\begin{Theorem}\label{theorem}Let $R$ be a ring and $\sigma$ an endomorphism of $R$ such that $R$ is $\sigma$-compatible. Then
\par$(1)$ $R$ is right idempotent reflexive if and only if $R[x,\sigma]$ is right idempotent reflexive if and only if $R[[x,\sigma]]$ is right idempotent reflexive.
\par$(2)$ $R$ is left idempotent reflexive if and only if $R[x,\sigma]$ is left idempotent reflexive if and only if $R[[x,\sigma]]$ is left idempotent reflexive.
\par$(3)$ $R$ is idempotent reflexive if and only if $R[x,\sigma]$ is idempotent reflexive if and only if $R[[x,\sigma]]$ is idempotent reflexive.
\end{Theorem}

\begin{proof}Obvious from Propositions \ref{prop1} and \ref{prop2}.
\end{proof}

Following Agayev et al. \cite{agayev/2009}. A ring $R$ is called $\sigma$-abelian if it is abelian and $\sigma$-compatible. In the next, we use the same idea to introduce $\sigma$-idempotent reflexive rings as a generalization of idempotent reflexive rings and $\sigma$-abelian rings.

\begin{definition}Let $R$ be a ring and $\sigma$ an endomorphism of $R$.
\NL$(1)$ $R$ is called right $($resp., left$)$ $\sigma$-idempotent reflexive if it is right $($resp., left$)$ idempotent reflexive and $\sigma$-compatible.
\NL$(2)$ $R$ is called $\sigma$-idempotent reflexive if it is  both right and left $\sigma$-idempotent reflexive.
\end{definition}

A ring $R$ is right (resp., left) idempotent reflexive if $R$ is right (resp., left) $id_R$- idempotent reflexive, where $id_R$ denotes the identity endomorphism of $R$. We can easily see that $\sigma$-abelian rings are $\sigma$-idempotent reflexive.

\begin{example}\label{ex1}Consider the ring $R=\set{\left(
\begin{array}{cc}
a & b \\
0 & a \\
\end{array}
\right)\mid a,b\in \Z_4}$, where $\Z_4$ is the ring of integers modulo $4$. Let $\sigma$ an endomorphism of $R$ defined by $$\sigma\parth{\left(
\begin{array}{cc}
a & b \\
0 & a \\
\end{array}
        \right)}=\left(
          \begin{array}{cc}
a & -b \\
0 & a \\
\end{array}
        \right)\mathrm{for\;all}\;\left(
          \begin{array}{cc}
a & b \\
0 & a \\
\end{array}
        \right)\in R$$
By \cite[Example 2.2]{agayev/2009}. The ring $R$ is abelian so it is right and left idempotent reflexive. Also, $R$ is $\sigma$-compatible. Thus $R$ is right and left $\sigma$-idempotent reflexive.
\end{example}

\begin{corollary}\label{corollary1}Let $R$ be a ring and $\sigma$ an endomorphism of $R$. Then
\par$(1)$ If $R$ is left $\sigma$-idempotent reflexive then $R[x,\sigma]$ and $R[[x,\sigma]]$ are left idempotent reflexive.
\par$(2)$ If $R$ is right $\sigma$-idempotent reflexive then $R[x,\sigma]$ and $R[[x,\sigma]]$ are right idempotent reflexive.
\par$(3)$ If $R$ is $\sigma$-idempotent reflexive then $R[x,\sigma]$ and $R[[x,\sigma]]$ are idempotent reflexive.
\end{corollary}

\begin{proof}Immediate from Proposition \ref{prop1}.
\end{proof}

\section{Abelian property}

In this section, $\sigma$ is an endomorphism of $R$ such that $\sigma(1)=1$. We can easily see that $Id(R)\subseteq Id(R[x])$. However $Id(R)\neq Id(R[x])$. We can take the next example.

\begin{example}Let $F$ be a field and consider the ring
 $R=\left(
\begin{array}{cc}
F & F \\
0 & F \\
\end{array}
\right)$, then
$e=\left(\begin{array}{cc}
1 & 0 \\
0 & 0 \\
\end{array}
\right)+\left( \begin{array}{cc}
0 & 1 \\
0 & 0 \\
\end{array}
\right)x$ is an element of $Id(R[x])$. But $e\not\in Id(R)$.
\end{example}

According to Kanwara et al. \cite[Corollary 6]{matzuk/2013}, a ring $R$ is abelian if and only if $R[x]$ is abelian. In the next, we give a generalization to skew polynomial rings and skew power series rings.

\begin{lemma}\label{lemma}Let $R$ be an abelian ring such that $\sigma(Re)\subseteq Re$ for all $e\in Id(R)$. For any $e=\sum_{i=0}^ne_ix^i\in Id(R[x,\sigma])$ $($resp., $e=\sum_{i=0}^{\infty}e_ix^i\in Id(R[[x,\sigma]])$$)$, we have $ee_0=e_0$ and $e_0e=e$.
\end{lemma}

\begin{proof}Let $e=\sum_{i=0}^ne_ix^i\in Id(R[x,\sigma])$, then we have the following system of equations
$$e_0^2=e_0\eqno(0)$$
$$e_0e_1+e_1\sigma(e_0)=e_1\qquad\qquad\;\quad\eqno(1)$$
$$e_0e_2+e_1\sigma(e_1)+e_2\sigma^2(e_0)=e_2\qquad\qquad\qquad\qquad\qquad\eqno(2)$$
$$\vdots$$
$$e_0e_n+e_1\sigma(e_{n-1})+\cdots+e_n\sigma^n(e_0)=e_n\qquad\qquad\qquad\qquad\qquad\qquad\;\eqno(n)$$
Since abelian rings are idempotent reflexive, then by the proof of Proposition \ref{prop1}, we have $e_i\in Re_0R=e_0R$ for all $0\leq i\leq n$, because $e_0$ is central. Hence $e_i=e_0e_i=e_ie_0$ for all $0\leq i\leq n$ and so $e_0e=e\;\;(1')$.
\NL On the other hand, from equations $(1)$ and $(1')$, we have $e_1\sigma(e_0)=0$ . Also, we have $e_1\sigma(e_1)+e_2\sigma^2(e_0)=0\;(2')$ from equations $(2)$ and $(1')$. But $\sigma^2(e_0)$ is an idempotent, then equation $(2')$ becomes $e_1\sigma(e_1)\sigma^2(e_0)+e_2\sigma^2(e_0)=e_1\sigma[e_1\sigma(e_0)]+e_2\sigma^2(e_0)=0$, since $e_1\sigma(e_0)=0$ we get $e_2\sigma^2(e_0)=0$. Continuing this procedure yields $e_i\sigma^i(e_0)=0$ for all $1\leq i\leq n$. Thus $ee_0=e_0$. With similar method we get the result for $R[[x,\sigma]]$.
\end{proof}

\begin{lemma}\label{lemma2}Let $R$ be a ring and $\sigma$ an endomorphism of $R$. If $R$ satisfies the condition $(\mathcal{C_{\sigma}})$ then for any $e^2=e\in R$, we have $\sigma(e)=e$.
\end{lemma}

\begin{proof}Let $e\in Id(R)$, by the condition $(\mathcal{C_{\sigma}})$ and from $\sigma(e)(1-\sigma(e))=(1-\sigma(e))\sigma(e)=0$, we get $\sigma(e)=\sigma(e)e$ and $e=\sigma(e)e$. Therefore $\sigma(e)=e$.
\end{proof}

\begin{proposition}\label{prop3}If $R$ satisfies the condition $(\mathcal{C_{\sigma}})$. Then $R$ is abelian if and only if $R[x,\sigma]$ is abelian if and only if $R[[x,\sigma]]$ is abelian.
\end{proposition}

\begin{proof}It suffices to show the necessary condition, let $e=\sum_{i=0}^ne_ix^i\in Id(R[x,\sigma])$. From Lemma \ref{lemma}, we get $e_i=e_ie_0$ and $e_i\sigma^i(e_0)=0$ for all $1\leq i\leq n$. By the condition $(\mathcal{C_{\sigma}})$, we have $e_i=0$ for all $1\leq i\leq n$, then $e=e_0\in Id(R)$. Therefore $R[x,\sigma]$ is abelian. The same method for $R[[x,\sigma]]$.
\end{proof}

\begin{example}\label{ex4}Let $F$ be a field. Consider the ring
$$R=\set{\left(
          \begin{array}{cccc}
            a & b & 0 & 0 \\
            0 & a & 0 & 0 \\
            0 & 0 & u & v \\
            0 & 0 & 0 & u \\
          \end{array}
        \right)\mid a,b,u,v\in F
}$$ and $\sigma$ an endomorphism of $R$ defined by $$\sigma\parth{\left(
          \begin{array}{cccc}
            a & b & 0 & 0 \\
            0 & a & 0 & 0 \\
            0 & 0 & u & v \\
            0 & 0 & 0 & u \\
          \end{array}
        \right)}=\left(
          \begin{array}{cccc}
            u & v & 0 & 0 \\
            0 & u & 0 & 0 \\
            0 & 0 & a & b \\
            0 & 0 & 0 & a \\
          \end{array}
        \right)\mathrm{for\;all}\;\left(
          \begin{array}{cccc}
            a & b & 0 & 0 \\
            0 & a & 0 & 0 \\
            0 & 0 & u & v \\
            0 & 0 & 0 & u \\
          \end{array}
        \right)\in R$$
\NL$(1)$ $R$ is abelian. But $R[x;\sigma]$ is not abelian by \cite[Example 2.17]{agayev/2009}.
\NL$(2)$ $R$ does not satisfy the condition $(\mathcal{C_{\sigma}})$. Let $e_{ij}$ denote the $4\times 4$ matrix units having alone $1$ as its $(i,j)$-entry and all other entries $0$.
\NL Consider $$e=\left(
          \begin{array}{cccc}
            1 & 0 & 0 & 0 \\
            0 & 1 & 0 & 0 \\
            0 & 0 & 0 & 0 \\
            0 & 0 & 0 & 0 \\
          \end{array}
        \right)\in Id(R)$$
Then
$$\sigma(e)=\left(
          \begin{array}{cccc}
            0 & 0 & 0 & 0 \\
            0 & 0 & 0 & 0 \\
            0 & 0 & 1 & 0 \\
            0 & 0 & 0 & 1 \\
          \end{array}
        \right)\neq e$$
By Lemma \ref{lemma2}, the condition $(\mathcal{C_{\sigma}})$ is not satisfied.
\end{example}

\begin{example}\label{ex3}Let $\Z_2$ is the ring of integers modulo $2$, take $R=\Z_2\oplus\Z_2$ with the usual addition and multiplication. Consider $\sigma\colon R\rightarrow R$ defined by $\sigma((a,b))=(b,a)$. Clearly $R$ is abelian because it is commutative. Let $e=(1,0)+(0,1)x\in Id(R[x,\sigma])$.  Since $(0,1)e=(0,1)x\neq e(0,1)=0$ then $e$ is not central. Thus $R[x,\sigma]$ is not abelian. Note that $R$ does not satisfy the condition $(\mathcal{C_{\sigma}})$, because $(0,1)\sigma(0,1)=0$ but $(0,1)^2\neq 0$.
\end{example}

We can see from Examples \ref{ex4} and \ref{ex3} that the condition $(\mathcal{C_{\sigma}})$ in Proposition \ref{prop3} is not superfluous.

\begin{corollary}\label{corollary}Let $R$ be a $\sigma$-compatible ring. Then  $R$ is abelian if and only if $R[x,\sigma]$ is abelian if and only if $R[[x,\sigma]]$ is abelian
\end{corollary}

\begin{proof}Obvious from Proposition \ref{prop3}.
\end{proof}

\begin{corollary}[{\cite[Corollary 6]{matzuk/2013}}]Let $R$ be a ring. Then
$R$ is abelian if and only if $R[x]$ is abelian if and only if $R[[x]]$ is abelian.
\end{corollary}

\begin{proof}It suffices to take $\sigma=id_R$ in Proposition \ref{prop3}.
\end{proof}

\begin{corollary}[{\cite[Theorem 2.21]{agayev/2009}}]\label{corollary2}If $R$ is $\sigma$-abelian then $R[x,\sigma]$ and $R[[x,\sigma]]$ are abelian.
\end{corollary}

\begin{proof}Clearly from Proposition \ref{prop3}.
\end{proof}

\begin{remark} If for a ring $R$ we have $Id(R)=\{0,1\}$, then $Id(R[x,\sigma])=\{0,1\}$. Indeed, let $e=\sum_{i=0}^ne_ix^i\in Id(R[x,\sigma])$ and consider the system cited in the proof of Lemma \ref{lemma}. If $e_0=0$ or $e_0=1$, we can easily see that $e_i=0$ for all $1\leq i\leq n$.
\end{remark}

The following example shows that the converse of Corollary \ref{corollary2} is not true.

\begin{example}\label{ex2}Let $\K$ be a field and $R=\K[t]$ a polynomial ring over $\K$
with the endomorphism $\sigma$ given by $\sigma(f(t))=f(0)$  for all
$f(t)\in R$. Since $R$ is commutative then it is abelian. We claim that $R$ is not $\sigma$-abelian. Indeed, take $f=a_0+a_1t+a_2t^2+\cdots+a_nt^n$ and
$g=b_1t+b_2t^2+\cdots+b_mt^m$, since $g(0)=0$ so, $f\sigma(g)=0$ but $fg\neq 0$. Therefore $R$ is not $\sigma$-compatible. Since $R$ has only two idempotents $0$ and $1$ then $Id(R[x,\sigma])=Id(R)$ and thus $R[x,\sigma]$ is abelian.
\end{example}



\end{document}